\documentclass[11pt,a4paper]{article}
\title{\bf Dimension free Harnack inequalities \\on $\RCD(K, \infty)$ spaces}
\author{Huaiqian Li\footnote{Email: huaiqianlee@gmail.com. Partially supported by the National Natural Science Foundation of China (NSFC) No.11401403 and the Australian Research Council (ARC) grant DP130101302. Current address: Department of Mathematics, Macquarie University, NSW 2109 Sydney, Australia; email: huaiqian.li@mq.edu.au.}\vspace{3mm}\\
{\footnotesize School of Mathematics, Sichuan University, Chengdu 610064, P. R. China}
}
\date{}

\usepackage{amssymb,amsmath,amsfonts,amsthm,color,mathrsfs}

\def\R{\mathbb{R}}

\def\d{\textup{d}}

\def\BE{\textup{BE}}

\def\CD{\textup{CD}}
\def\Ch{\textup{Ch}}

\def\Hess{\textup{Hess}}
\def\Lip{\textup{Lip}}
\def\Ric{\textup{Ric}}

\def\RCD{\textup{RCD}}
\def\<{\langle}
\def\>{\rangle}
\def\Proof.{\noindent{\bf Proof. }}

\def\Ent{\textup{Ent}}

\def\newdot{{\kern.8pt\cdot\kern.8pt}}

\newtheorem{theorem}{Theorem}[section]
\newtheorem{lemma}[theorem]{Lemma}
\newtheorem{corollary}[theorem]{Corollary}
\newtheorem{proposition}[theorem]{Proposition}

\newtheorem{definition}[theorem]{Definition}
\theoremstyle{definition}\newtheorem{remark}[theorem]{Remark}

\begin{document}

\maketitle
\makeatletter 
\renewcommand\theequation{\thesection.\arabic{equation}}
\@addtoreset{equation}{section}
\makeatother 

\begin{abstract}
The dimension free Harnack inequality for the heat semigroup is established on the $\RCD(K,\infty)$ space, which is a non-smooth metric measure space having the Ricci curvature bounded from below in the sense of Lott--Sturm--Villani  plus the Cheeger energy being quadratic. As its applications, the heat semigroup entropy-cost inequality and contractivity properties of the semigroup are studied, and a strong enough Gaussian concentration implying the log-Sobolev inequality is also shown as a generalization of the one on the smooth Riemannian manifold.
\end{abstract}



\section{Introduction}
Let $M$ be a finite dimensional complete and connected smooth Riemannian manifold with boundary either empty or convex and Riemannian distance $d$, and let $L=\Delta+\nabla V$ for some $V\in C^2(M)$ with the corresponding diffusion semigroup denoted by $\{P_t\}_{t\geq 0}$. Let
\begin{eqnarray}\label{classical BE}
\Ric-\Hess_V\geq K,
\end{eqnarray}
where $\Ric$ is the Ricci curvature, $\Hess_V$ is the Hessian of $V$ and $K\in \R$. By the Bochner formula, it is well known that \eqref{classical BE} is equivalent to the classical Bakry--Emery curvature-dimension condition, with infinite dimension, for $L$, proposed in the seminal work \cite{BakryEmery} in 1983.

In 1997, Wang \cite{Wang1997} introduced the dimension free Harnack inequality for the diffusion semigroup $\{P_t\}_{t\geq 0}$ under the assumption \eqref{classical BE}, which can be formulated as
\begin{eqnarray}\label{wangharnack}
|\left(P_t f\right)(x)|^p\leq \left(P_t |f|^p\right)(y) \exp\left\{\frac{p K d(x,y)^2}{2(p-1)\left(e^{2Kt}-1\right)}\right\},
\end{eqnarray}
for all $f\in C_b(M)$, $x,y\in M$ and $p>1$.

The Harnack inequality \eqref{wangharnack} has been extensively investigated and applied successfully to the study of functional inequalities, heat kernel estimates, transportation-cost inequalities, short time behavior of transition probabilities and so on. See \cite{Wang2005,Wang2013a} and references therein. One of the remarkable feature of the Harnack inequality \eqref{wangharnack}, as its name conveys, is its independence of the dimension of the diffusion operator, and hence it can be established for infinite dimensional operators.

Recall the original proof in \cite{Wang1997}, the pointwise $L^1$-gradient estimate of the operator $P_t$ in the form
\begin{eqnarray}\label{pointwiseL1}
|\nabla P_t f|\leq C(t) P_t|\nabla f|, \quad \mbox{for all }t>0\mbox{ and some constant } C(t)>0,
\end{eqnarray}
is crucial for the establishment of the Harnack inequality \eqref{wangharnack} on the smooth Riemannian manifold $M$. It is well known that in the smooth Riemannian context, the pointwise $L^1$-gradient estimate \eqref{pointwiseL1} is equivalent to the curvature-dimension condition \eqref{classical BE} and also equivalent to the Harnack inequality \eqref{wangharnack} (see e.g. \cite{Wang2005,Wang2013a}). However, unfortunately, in general, such as on the smooth sub-Riemannian manifold, the pointwise $L^1$-gradient estimate is not easy to get, not even on the non-smooth space. To the authors knowledge, only on the Heisenberg (type) group the pointwise $L^1$-gradient estimate is established for all $t>0$ (see \cite{BakryBaudoinBonnefontChafai,Eldredge2010,Li2006}). Recently, Baudoin--Garofalo \cite{BaudoinGarofalo} introduced the generalized curvature-dimension condition with respect to (w.r.t. for short) a class of smooth second order diffusion operator on the smooth connected manifold, which is the generalization of the classical Bakry--Emery curvature-dimension condition in the sub-Riemannian context. Under the generalized curvature-dimension condition, the dimension free Harnack inequality was established and also other functional inequalities (see \cite{BaudoinBonnefont} and the more general curvature-dimension condition case see \cite{Wang2012}).

From now on, we assume that $(X,d,\mu)$ is a metric measure space, which means that $(X,d)$ is a Polish space and $\mu$ is a nonnegative $\sigma$-finite Borel measure finite on metric balls.

In the pioneer works of Lott--Villani \cite{LottVillani2009} and Sturm  \cite{Sturm2006a,Sturm2006b}, via the convexity of the entropy functional, a notion of Ricci curvature bounded from below by $K$ and dimension bounded above by $N$ on the metric measure space $(X,d,\mu)$, denoted by $\CD(K,N)$ spaces, was proposed independently by the authors (note only the case $K=0$ or $N=\infty$ were considered in \cite{LottVillani2009} and see \cite{Villani2009} for an elaborate presentation). Here $K\in \R$ and $N\in [1,\infty]$. The crucial properties of this notion are the compatibility with the classical Bakry--Emery curvature-dimension condition on the smooth Rimannian manifold and the stability w.r.t. the measured Gromov--Hausdorff convergence. A strengthening of the Lott--Sturm--Villani $\CD(K,\infty)$ condition  is the so-called Riemannian curvature bounded from below, denoted by  $\RCD(K,\infty)$ (see Definition \ref{RCDspace} below), introduced by Ambrosio--Gigli--Savar\'{e} \cite{AmbrosioGigliSavare2011b} for the case that the reference measure $\mu$ is a probability measure (also see \cite{AmbrosioGigliMondinoRajala} for the $\sigma$-finite measure case). This notion is also stable under the measured  Gromov--Hausdorff convergence and rules out Finsler manifolds (since the linearity of the heat flow is required), and hence it is useful for describing the closure of the class of Riemannian manifolds with Ricci curvature bounded from below. Examples of the $\RCD(K,\infty)$ space include complete and finite dimensional spaces with Alexandrov curvature bounded from below (see \cite{GigliKuwadaOhta2013,ZhangZhu2010}), Euclidean spaces endowed with the Lebesgue measure, complete Riemannian manifolds with bounded geometry and limit spaces of smooth Riemannian manifolds with Ricci curvature bounded from below by $K$ (see \cite{CheegerColding1997,CheegerColding2000a,CheegerColding2000b}).

We should mention that, by the powerful calculus on the metric measure space $(X,d,\mu)$, under the $\CD(K,\infty)$ condition, Ambrosio--Gigli--Savar\'{e} \cite{AmbrosioGigliSavare2011a} proved the equivalence of the $L^2$-gradient flow generated by the Cheeger energy (as the role of the Dirichlet form plays) and the Wasserstein gradient flow of the relative entropy, which allows to call either of the gradient flows as a heat flow (see also \cite{KoskelaZhou2012} for a different starting point from a Dirichlet form on a measure space endowed with the distance induced by the Dirichlet form).  Recently, they introduced a weak Barkry--Emery condition $\BE(K,N)$ (see Definition \ref{BE} below), which is a proper integral form of the classic Bakry--Emery curvature-dimension condition, and showed that $\BE(K,\infty)$ is equivalent to $\CD(K,\infty)$ under the assumption of the Cheeger energy being quadratic (see \cite{AmbrosioGigliSavare2012}). This result fills the gap between the two approaches of Bakry--Emery and Lott--Sturm--Villani to the notion of Ricci curvature bounded from below.

We emphasize that, by the self-improvement property of the $\BE(K,\infty)$ condition, Savar\'{e} \cite{Savare2013} proved that
\begin{eqnarray}\label{L1gradient}
\sqrt{\Gamma(P_tf)}\leq e^{-Kt}P_t \sqrt{\Gamma(f)},\quad\mu\mbox{-a.e. in }X,
\end{eqnarray}
where $\{P_t\}_{t\geq 0}$ is the semigroup generated by the heat flow. It can be considered as the candidate of the point-wise $L^1$-gradient estimate \eqref{pointwiseL1}. Here $\Gamma$ is the \emph{carr\'{e} du champ} operator. See Section 2.3 for details. Note that \eqref{L1gradient} is crucial for the establishment of the dimension free Harnack inequality in our framework (see Theorem \ref{harnack} below).

Recently, Erbar--Kuwada--Sturm \cite{ErbarKuwadaSturm2013} introduced the the notion of $\RCD^*(K,N)$ spaces. The Li--Yau gradient estimate, Harnack type inequality and Bakry--Qian inequality for the heat flow were proved in \cite{GarofaloMondino} on the $\RCD^*(K,N)$ spaces $(X,d,\mu)$ with $\mu$ being a probability measure. But it seems not easy to generalize their results to the case that $\mu$ can be infinite (e.g. $\mu$ is a Borel measure), since the arguments depend heavily on $\mu$ being a probability measure.  As well, the log-Harnack inequality with dimension and its applications on transportation-cost inequalities on the $\RCD^*(K,N)$ space were considered by the author \cite{Lihuaiqian}. Recently, in a submitted paper \cite{JLZ2014} joint with R. Jiang and H. Zhang, sharp upper and low bounds on the heat kernel and its gradient were established in $\RCD^*(K,N)$ spaces, as well as some interesting applications, including the large time asymptotics of the heat kernel, stability of solutions to the heat equation, and the boundedness of (local) Riesz transforms.

The aim of this paper is to establish the dimension free Harnack inequality on non-smooth metric measure spaces with Riemannian curvature bounded from below, namely the so-called $\RCD(K,\infty)$ spaces. In Section 2, we present some preliminaries, especially the notion of $\RCD(K,\infty)$ spaces, and recall some known results such as the $L^1$-gradient estimate of the heat flow. In Section 3, we show the main result, Theorem \ref{harnack}, and its proof. Finally, as applications of the dimension free Harnack inequality, some of its consequences such as the log-Harnack inequality, entropy-cost inequality, contractivity properties of the semigroup and the strong enough Gaussian concentration implying the log-Sobolev inequality are shown in Section 4. Note that the main result seems new, but the idea of the proof is not.

\section{Preliminaries}
In this section, we introduce some necessary notations, definitions and recall some known results, following closely the recent papers \cite{AmbrosioGigliSavare2011a,AmbrosioGigliSavare2011b,AmbrosioGigliSavare2012,Savare2013}.

Let $(X,d)$ be a Polish space endowed with a nonnegative Borel measure $\mu$ with support $X$ and satisfying
\begin{eqnarray}\label{locallyfinite}
\mu(B(x,r))<\infty,\quad\mbox{for any }x\in X\mbox{ and }r>0,
\end{eqnarray}
where $B(x,r)$ is the ball of radius $r$ centered at $x$ in $X$ w.r.t. the metric $d$. We denote by $C(X)$ ($C_b(X)$) the (bounded) continuous functions on $X$, by $\Lip(X)$ ($\Lip_b(X)$) the space of (bounded) Lipschitz continuous functions on $X$, by $\mathcal{P}(X)$ the space of Borel probability measures on $(X,d)$ and by $\mathcal{P}_p(X)$, $p\in [1,\infty)$, the subset of $\mathcal{P}(X)$ with finite $p$-th moment, i.e., $\nu\in \mathcal{P}(X)$ such that
$$\int_X d(o,\cdot)^p\,\d\nu<\infty,\quad\mbox{for some (hence any) }o\in X.$$

\subsection{Gradient flows and Dirichlet forms}
Given a closed interval $I\subset\R$, let $AC^p(I;X)$, $p\in [1,\infty]$, be the set of all the absolutely continuous curves $\gamma: I\rightarrow X$ such that for some $g\in L^p(I)$, it holds
\begin{eqnarray}\label{ac}
d(\gamma(s),\gamma(t))\leq \int_s^t g(r)\,\d r,\quad\mbox{for any }s,t\in I,\,\ s<t.
\end{eqnarray}
It is true that, if $\gamma\in AC^p(I;X)$, then the metric slope
$$\lim_{\delta\rightarrow 0}\frac{d(\gamma(r+\delta),\gamma(r))}{|\delta|},$$
denoted by $|\dot{\gamma}|(r)$, exists for $\mathcal{L}^1$-a.e. $r\in I$, belongs to $L^p(I)$, and it is the minimal function $g$ such that \eqref{ac} holds (see \cite[Theorem 1.1.2]{AmbrosioGigliSavare2005} for the proof). The length of the absolutely continuous curve $\gamma: [0,1]\rightarrow X$ is defined by $\int_0^1 |\dot{\gamma}|(r)\,\d r$.  We say that an absolutely continuous curve $\gamma: [0,1]\rightarrow X$ has constant speed if $|\dot{\gamma}|(t)$ is a constant, and it is a geodesic if
$$d(\gamma(s),\gamma(t))=|t-s|d(\gamma(0),\gamma(1)),\quad\mbox{for any }s,t\in [0,1].$$
The local Lipschitz constant (or metric slope) of a function $f\in \Lip_b(X)$ is defined by
$$|Df|(x)=\limsup_{y\rightarrow x}\frac{|f(y)-f(x)|}{d(y,x)}.$$

For any $f\in L^2(X,\mu)$, define the Cheeger energy functional $\Ch: L^2(X,\mu)\rightarrow [0,\infty]$ by
\begin{eqnarray}\label{Cheeger}
\Ch(f)=\inf\left\{\liminf_{n\rightarrow \infty} \frac{1}{2}\int_X |Df_n|^2\,\d\mu: f_n\in \Lip_b(X),\,\ \|f_n - f\|_{L^2(X,\mu)}\rightarrow 0\right\}.
\end{eqnarray}
It is immediate to check that $\Ch$ is convex and lower semicontinuous in $L^2(X,\mu)$ (see e.g. \cite[Theorem 4.5]{AmbrosioGigliSavare2011a}), with domain dense in $L^2(X,\mu)$. So the classical theory of gradient flows in Hilbert spaces (see e.g. \cite{AmbrosioGigliSavare2005}) guarantees that for any $f\in L^2(X,\mu)$ there exists a unique gradient flow for $\Ch$ starting from $f$, which yields a semigroup on $L^2(X,\mu)$, denoted by $\{P_t\}_{t\geq 0}$. However, $\Ch$ is in general not necessarily a quadratic form.  The domain of $\Ch$ is the Sobolev space $W^{1,2}(X,d,\mu)$ defined by
$$W^{1,2}(X,d,\mu)=\{f\in L^2(X,\mu): \Ch(f)<\infty\},$$
endowed with the norm $\|f\|_{W^{1,2}}:=\big(\|f\|^2_{L^2(X,\mu)}+2\Ch(f)\big)^{1/2}$.

Note that $(W^{1,2}(X, d,\mu), \|\cdot\|_{W^{1,2}})$ is always a Banach space, while, in general, it is not a Hilbert space since $\Ch$ is not a quadratic form on $L^2(X,\mu)$.

Recall that the important properties of the semigroup $\{P_t\}_{t\geq 0}$ are that it is 1-homogeneous, i.e., $P_t(\lambda f)=\lambda P_t(f)$ for any $\lambda\in \R$ and $t\geq 0$, and it satisfies the maximum principle: for $f\in L^2(X,\mu)$, if $f\leq c$ (resp. $f\geq c$) $\mu$-a.e. in $X$ for some $c\in \R$, then $P_t(f)\leq c$ (resp. $P_t(f)\geq c$) $\mu$-a.e. in $X$ for any $t\geq 0$ (see \cite[Theorem 4.16]{AmbrosioGigliSavare2011a}).

Recall also that in the case that $\Ch$ is a quadratic form on $L^2(X,\mu)$, since $\Ch$ is 2-homogeneous and convex, this property is equivalent to the parallelogram rule, i.e,
\begin{eqnarray}\label{parallelogram}
\Ch(f+g)+\Ch(f-g)=2\Ch(f)+2\Ch(g),\quad\mbox{for every }f,g\in L^2(X,\mu).
\end{eqnarray}
Then we denote by $\mathcal{E}$ the Dirichlet form on $L^2(X,\mu)$ associated to $\Ch$ with domain $W^{1,2}(X,d,\mu)=:\mathcal{D}(\mathcal{E})$, which is a Hilbert space w.r.t. the norm $\|\cdot\|_{W^{1,2}}$ and $\Lip(X)$ is dense in it. In other words, $\mathcal{E}: \mathcal{D}(\mathcal{E})\times \mathcal{D}(\mathcal{E})\rightarrow \R$ is the unique bilinear symmetric form satisfying
\begin{eqnarray}\label{CheegerDirichlet}
\mathcal{E}(f,f)=2\Ch(f),\quad\mbox{for any }f\in W^{1,2}(X,d,\mu).
\end{eqnarray}
For $f,g\in W^{1,2}(X,d,\mu)$, $$\mathcal{E}(f,g):=\frac{1}{4}\big(\mathcal{E}(f+g,f+g)-\mathcal{E}(f-g,f-g) \big).$$
Good references for the theory of Dirichlet forms are \cite{BouleauHirsch,MaRochner,FukushimaOshimaTakeda2011}.

In the quadratic case, moreover, $\{P_t\}_{t\geq 0}$ is a semigroup of bounded self-adjoint linear operators on $L^2(X,\mu)$ and its generator is denoted by $L$ with domain $\mathcal{D}(L)$ dense in $\mathcal{D}(\mathcal{E})$. As mentioned above, $\{P_t\}_{t\geq 0}$ is the unique analytic Markov semigroup such that for any $f\in L^2(X,\mu)$, the curve $t\mapsto P_tf$ from $[0,\infty)$ to $L^2(X,\mu)$ satisfies $\lim_{t\downarrow 0} P_tf=f$ in $L^2(X,\mu)$, and for any $t>0$,
$$P_tf\in \mathcal{D}(L),\quad \frac{\d}{\d t}P_tf=LP_tf.$$

\subsection{$\RCD(K,\infty)$ spaces}

For $\mu_1,\mu_2\in\mathcal{P}_p(X)$ and $p\in [1,\infty)$, the $L^p$-Wasserstein distance $W_p(\mu_1,\mu_2)$ is defined by
\begin{eqnarray}\label{wasserstein}
W_p(\mu_1,\mu_2)=\inf_{\pi} \left\{\int_{X\times X} d(x,y)^p\,\d\pi(x,y)\right\}^{\frac{1}{p}},
\end{eqnarray}
where the infimum is taken among all $\pi\in \mathcal{P}(X\times X)$ with $\mu_1$ and $\mu_2$ being the first and the second marginal distributions of $\pi$, respectively. $\pi$ is called a coupling of $\mu_1$ and $\mu_2$. Since the cost $d^p$ with $p\in [1,\infty)$ is lower semi-continuous, the infimum in \eqref{wasserstein} is attained. Every coupling achieving the infimum is called an optimal one. See e.g. \cite{Villani2009}.

The relative entropy functional $\Ent_\mu: \mathcal{P}_2(X)\rightarrow [-\infty,\infty]$ w.r.t. $\mu$ is defined by (see \cite[Section 4.1]{Sturm2006a})
\begin{equation*}
\Ent_\mu(\rho)=
\begin{cases}
\lim_{\epsilon\downarrow 0}\int_{\{f>\epsilon\}} f\log f\,\d\mu,\quad &{\hbox{if}}\,\ \rho=f\mu,\\
\infty,\quad &{\hbox{otherwise}}.
\end{cases}
\end{equation*}
It coincides with $\int_{\{f>0\}}f\log f\,\d\mu$ and belongs to $[-\infty,\infty)$, provided $(f\log f)^+:=\max\{f\log f, 0\}$ is integrable w.r.t. $\mu$, and it is equal to $\infty$ otherwise. In particular, when $\mu\in \mathcal{P}(X)$, by Jensen's inequality, the functional $\Ent_\mu$ is nonnegative. The domain of the relative entropy, denoted by $\mathcal{D}(\Ent_\mu)$, is the set of $\rho\in \mathcal{P}_2(X)$ such that $\Ent_\mu(\rho)<\infty$.
\begin{definition}\label{RCDspace}
Let $K\in\R$ and let $(X,d)$ be a Polish space endowed with a nonnegative $\sigma$-finite Borel measure $\mu$ with support $X$ and satisfying \eqref{locallyfinite}. We say that $(X,d,\mu)$ is a $\CD(K,\infty)$ space if, for every couple of measures $\eta_0,\eta_1\in \mathcal{D}(\Ent_\mu)$, there exists a geodesic $\{\eta_t\}_{0\leq t\leq 1} \subset \mathcal{P}_2(X)$ such that
$$\Ent_\mu(\eta_t)\leq (1-t)\Ent_\mu(\eta_0)+t\Ent_\mu(\eta_1)-\frac{K}{2}t(1-t)W_2(\eta_0,\eta_1)^2,\quad\mbox{for any }t\in [0,1].$$
In addition, if $\Ch$ is a quadratic form on $L^2(X,\mu)$ according to \eqref{parallelogram}, then we say that $(X,d,\mu)$ has the Riemannian curvature bounded from below by $K$, and denote it by $\RCD(K,\infty)$ space for short.
\end{definition}

Notice that, if  $(X,d,\mu)$ is a $\CD(K,\infty)$ space, then the reference measure $\mu$ always satisfies that for some point $o\in X$ and some constants $c_1>0$,
$c_2\geq 0$,
\begin{eqnarray*}\label{exp}
\mu(B(o,r))\leq c_1 e^{c_2r^2},\quad\mbox{for any }r>0,
\end{eqnarray*}
and hence for $\rho\in \mathcal{P}_2(X)$, $\Ent_\mu(\rho)$ can be bounded from below in terms of the second order moment of $\rho$ by the change of the reference measure formula (see (2.5) in \cite{AmbrosioGigliMondinoRajala}), and $(X,d)$ is a length space, i.e., for every $x_0,x_1\in X$,
\begin{eqnarray*}\label{length}
d(x_0,x_1)=\inf\left\{\int_0^1 |\dot{\gamma}|(r)\,\d r: \gamma\in AC([0,1];X), \gamma(i)=x_i, i=0,1\right\}.
\end{eqnarray*}
See \cite{Sturm2006a} for these facts.

Recall also that, by \cite[Theorem 9.3]{AmbrosioGigliSavare2011a} or \cite[Theorem 6.2]{AmbrosioGigliMondinoRajala}, if  $(X,d,\mu)$ is a $\CD(K,\infty)$ space, then for every $f\in L^2(X,\mu)$ such that $\nu=f\mu\in \mathcal{D}(\Ent_\mu)$, the $W_2$-gradient flow of $\Ent_\mu$ starting from $\nu$, denoted by $[0,\infty)\ni t\mapsto h_t\nu$, and the $L^2$-gradient flow of $\Ch$ starting from $f$ conincides, i.e.,
\begin{eqnarray}\label{identifyofgradientflows}
(P_tf)\mu=h_t(f\mu),\quad\mbox{for all } t\geq 0,
\end{eqnarray}
and hence we call either of the gradient flows a heat flow and the induced semigroup a heat semigroup; moreover, the heat semigroup $\{P_t\}_{t\geq 0}$ has the mass preserving property w.r.t. $\mu$ (see \cite[Theorem 4.20]{AmbrosioGigliSavare2011a}) in the sense that
\begin{eqnarray}\label{mass preserving}
\int_X P_t f\,\d\mu=\int_X f\,\d\mu,\quad\mbox{for every }f\in L^2(X,\mu).
\end{eqnarray}

\subsection{$L^1$-gradient estimates of the heat semigroup}
Let $\mathcal{E}: L^2(X,\mu)\rightarrow [0,\infty]$ be a strongly local, symmetric Dirichlet form with domain $\mathcal{D}(\mathcal{E}):=\{f\in L^2(X,\mu): \mathcal{E}(f)<\infty\}$ dense in $L^2(X,\mu)$ and assume $\mathcal{E}$ generates a mass preserving Markov semigroup $\{P_t\}_{t\geq 0}$ in $L^2(X,\mu)$ with generator $L$ and domain $\mathcal{D}(L)$ dense in $\mathcal{D}(\mathcal{E})$.

Assume $\mathcal{E}$ admits a \emph{carr\'{e} du champ} operator $\Gamma$, which is a bilinear, continuous and symmetric map, denoted by $\Gamma:\mathcal{D}(\mathcal{E})\times \mathcal{D}(\mathcal{E})\rightarrow L^1(X,\mu)$, and it is uniquely characterized in the algebra $\mathcal{D}_\infty(\mathcal{E}):=\mathcal{D}(\mathcal{E})\cap L^\infty(X,\mu)$ by
$$\int_X \Gamma(f,g)\phi\,\d\mu=\frac{1}{2}\left(\mathcal{E}(f,g\phi) + \mathcal{E}(g,f\phi)-\mathcal{E}(fg,\phi)\right),$$
for any $f,g,\phi\in \mathcal{D}_\infty(\mathcal{E})$. We write $\Gamma(f)=\Gamma(f,f)$ for short.

Let
$$D_{\mathcal{E}}(L)=\{f\in \mathcal{D}(L): Lf\in \mathcal{D}(\mathcal{E})\},$$
and
$$D_{L^\infty}(L)=\{\phi\in \mathcal{D}(L)\cap L^\infty(X,\mu): L\phi\in L^\infty(X,\mu)\}.$$
Following \cite{AmbrosioGigliSavare2012}, we introduce the trilinear form ${\bf{\Gamma}}_2$ by defining
$${\bf{\Gamma}}_2[f,g](\phi)=\frac{1}{2}\int_X\left(\Gamma(f,g)L\phi-\Gamma(f,Lg)\phi-\Gamma(Lf,g)\phi\right)\,\d\mu,$$
for every $(f,g,\phi)\in D({\bf{\Gamma}}_2)$, where $D({\bf{\Gamma}}_2):=D_{\mathcal{E}}(L)\times D_{\mathcal{E}}(L)\times D_{L^\infty}(L)$. We set ${\bf{\Gamma}}_2[f](\phi)={\bf{\Gamma}}_2[f,f](\phi)$ for short.

Inspired by \cite{BakryEmery,Bakry1997}, Ambrosio--Gigli--Savar\'{e} \cite{AmbrosioGigliSavare2012} introduced a weak form of the Bakry--Emery curvature-dimension condition recently.
\begin{definition}\label{BE}
Let $K\in \R$ and $N\in [1,\infty]$. We say that a strongly local Dirichlet form $\mathcal{E}$ on $L^2(X,\mu)$ satisfies a weak curvature-dimension condition, denoted by $\BE(K,N)$, if it admits a \emph{carr\'{e} du champ} operator $\Gamma$ and
$${\bf{\Gamma}}_2[f](\phi)\geq K\int_X \Gamma(f)\phi\,\d\mu+\frac{1}{N}\int_X (Lf)^2\phi\,\d\mu,$$
for every $(f,\phi)\in D_{\mathcal{E}}(L)\times D_{L^\infty}(L)$ with $\phi\geq 0$. In particular, for $N=\infty$, $\BE(K,N)$ is denoted by $\BE(K,\infty)$, i.e.,
$${\bf{\Gamma}}_2[f](\phi)\geq K\int_X \Gamma(f)\phi\,\d\mu,$$
for every $(f,\phi)\in D_{\mathcal{E}}(L)\times D_{L^\infty}(L)$ with $\phi\geq 0$.
\end{definition}

From \cite[Corollary 2.3]{AmbrosioGigliSavare2012}, $\BE(K,\infty)$ is equivalent to the following inequality:
\begin{eqnarray}\label{L2gradient}
\Gamma(P_tf)\leq e^{-2Kt}P_t\Gamma(f),\quad\mu\mbox{-a.e. in }X,\mbox{ for every }t>0,\,\ f\in \mathcal{D}(\mathcal{E}).
\end{eqnarray}

Inequality \eqref{L2gradient} is a kind of the $L^2$-gradient estimate of the heat semigroup $P_t$. Following the idea of Proposition 2.3 in \cite{Bakry1997} in the classical Bakry--Emery curvature-dimension condition context, Savar\'{e} \cite[Corollary 3.5]{Savare2013} proved a kind of $L^1$-gradient estimate of $P_t$ by the self-improvement property of the weak curvature-dimension condition $\BE(K,\infty)$. We present the result in the next lemma, which is the key to establish the dimension free Harnack inequality.
\begin{lemma}\label{SavareL1estimate}
Let $(\mathcal{E}, \mathcal{D}(\mathcal{E}))$ be a strongly local and quasi-regular symmetric Dirichlet form on $L^2(X,\mu)$. If $\BE(K,\infty)$ holds with $K\in \R$, then for every $f\in \mathcal{D}(\mathcal{E})$ and $t>0$,
$$\sqrt{\Gamma(P_tf)}\leq e^{-Kt}P_t\sqrt{\Gamma(f)},\quad \mu\mbox{-a.e. in }X.$$
\end{lemma}

Now let $(X,d,\mu)$ be a $\RCD(K,\infty)$ space and let $\mathcal{E}$ be the symmetric Dirichlet form induced by the Cheeger energy according to \eqref{CheegerDirichlet}. It turns out that $\mathcal{E}$ is strongly local (see \cite[Theorem 3.17]{AmbrosioGigliSavare2012}), quasi-regular (see the proof of \cite[Theorem 4.1]{Savare2013}), and it admits a \emph{carr\'{e} du champ} operator $\Gamma$.

Indeed, with the operator $\Gamma$, there is an intrinsic way to define a pseudo metric on $X$ by
$$d_{\mathcal{E}}(x,y)=\sup\{|\psi(x)-\psi(y)|: \psi\in \mathcal{D}(\mathcal{E})\cap C(X), \Gamma(\psi)\leq 1\,\ \mu\mbox{-a.e. in }X \},$$
for any $x,y\in X$. Note that the main result in \cite[Theorem 4.17]{AmbrosioGigliSavare2012} shows that if $(X,d,\mu)$ is a $\RCD(K,\infty)$ space, then the weak curvature-dimension condition $\BE(K,\infty)$ holds for the Dirichlet form $(\mathcal{E}, \mathcal{D}(\mathcal{E}))$, and \cite[Theorem 3.9]{AmbrosioGigliSavare2012} implies $d_{\mathcal{E}}$ is a metric on $X$ such that
$$d_{\mathcal{E}}(x,y)=d(x,y),\quad\mbox{for any }x,y\in X.$$
Hence we can work indifferently with either the metric $d$ or $d_{\mathcal{E}}$.

\section{Main results and proofs}

From now on, let $(X,d,\mu)$ be a $\RCD(K,\infty)$ space with  $K\in \R$. Then, there is an one-parameter family of maps $h_t: \mathcal{P}_2(X)\rightarrow \mathcal{P}_2(X)$ such that,
\begin{eqnarray}\label{w2contraction}
W_2(h_t\nu_1,h_t\nu_2)\leq e^{-Kt}W_2(\nu_1,\nu_2),\quad\mbox{for every }\nu_1,\nu_2\in \mathcal{P}_2(X)\mbox{ and }t\geq 0.
\end{eqnarray}
It can be shown that $ h_t\nu\ll \mu$ for any $\nu\in \mathcal{P}(X)$ and $t>0$, and $\{P_t\}_{t\geq 0}$ can be uniquely extended to a continuous semigroup on $L^1(X,\mu)$, still denoted by $\{P_t\}_{t\geq 0}$, such that \eqref{identifyofgradientflows} holds for any $f\in L^1(X,\mu)$ with $f\mu\in \mathcal{P}_2(X)$.

For every bounded or nonnegative Borel function $f$, define
$$\tilde{P}_tf(x)=\int_X f(y)\,\d (h_t\delta_x)(y),$$
where $\delta_x$ is the Dirac measure at $x\in X$. Then $\tilde{P}_tf$ is a version of $P_tf$ for all $f\in L^2(X,\mu)$ and an extension of $P_t$ to a continuous contraction semigroup in $L^1(X,\mu)$. In addition, for every  $f\in L^2\cap L^\infty(X,\mu)$, $\tilde{P}_tf$ is pointwise everywhere defined, the map $(t,x)\mapsto \tilde{P}_tf(x)$ belongs to $C_b\left((0,\infty)\times X\right)$ and, furthermore, $\tilde{P}_tf\in \Lip_b(X)$ for every $t>0$. See \cite[Theorems 7.1 and 7.3]{AmbrosioGigliMondinoRajala} and \cite[Theorem 3.17]{AmbrosioGigliSavare2012} for these and more properties of $\tilde{P}_t$.

Now we present the main result in the next theorem and then give a proof of it.
\begin{theorem}\label{harnack}
Let $(X,d,\mu)$ be a $\RCD(K,\infty)$ space with  $K\in \R$ and let $p>1$. For any $f\in L^1(X,\mu)+L^\infty(X,\mu)$, $t>0$, $\epsilon\in [0,1]$ and $x,y\in X$, the Harnack inequality
\begin{eqnarray}\label{harnack1}
\big[\big|\tilde{P}_tf\big|+\epsilon\big]^p(x)\leq \tilde{P}_t\left[(|f|+\epsilon)^p\right](y)\exp\left\{\frac{p K d(x,y)^2}{2(p-1)\left(e^{2Kt}-1\right)}\right\}
\end{eqnarray}holds. In particular,
\begin{eqnarray}\label{harnack2}
\big|\big(\tilde{P}_tf\big)(x)\big|^p\leq \big(\tilde{P}_t |f|^p\big)(y)\exp\left\{\frac{p K d(x,y)^2}{2(p-1)\left(e^{2Kt}-1\right)}\right\}.
\end{eqnarray}
\end{theorem}

Here $f\in L^1(X,\mu)+L^\infty(X,\mu)$ means that $f$ can be written as $f=g_1+g_2$ such that $g_1\in L^1(X,\mu)$ and $g_2\in L^\infty(X,\mu)$. The idea of the proof is from \cite[Lemma 2.1]{Wang1997}. But some efforts are needed to make a careful modification of the original proof in the aforementioned reference to adapt to our more abstract context.

We first borrow a lemma from \cite[Lemma 4.5]{AmbrosioGigliSavare2012} and omit its proof here.
\begin{lemma}\label{lemmaofharnack1}
Let $(X,d,\mu)$ be a $\RCD(K,\infty)$ space with  $K\in \R$. Assume that $\theta: [0,\infty)\rightarrow \R$ is a twice continuously differentiable function. Let $f\in \Lip_b(X)\cap \mathcal{D}(\mathcal{E})$ be nonnegative and let $\nu\in \mathcal{P}(X)$. Then the function
$$G(s):=\int_X \theta(P_{t-s}f)\,\d (h_s\nu),\quad s\in [0,t]$$
belongs to $C([0,t])\cap C^1((0,t))$, and for any $s\in (0,t)$, it holds
\begin{eqnarray}\label{lemma-1}
G'(s)=\int_X \theta''(P_{t-s}f)\Gamma(P_{t-s}f)\,\d (h_s\nu).
\end{eqnarray}
\end{lemma}

The next lemma is a slight modification of \cite[Proposition 3.11]{AmbrosioGigliSavare2012} in the $\RCD(K,\infty)$ space case and its proof is the same as the one there. So we also omit it.
\begin{lemma}\label{lemmaofharnack2}
Let $(X,d,\mu)$ be a $\RCD(K,\infty)$ space with  $K\in \R$. If, for $g\in C_b(X)\cap \mathcal{D}(\mathcal{E})$, a bounded upper semicontinuous function $\xi: X\rightarrow [0,\infty)$ satisfies $\sqrt{\Gamma(g)}\leq \xi$ $\mu$-a.e. in $X$, then $|Dg|(x)\leq \xi(x)$ for every $x\in X$.
\end{lemma}

With the above preparation at hand, we can prove Theorem \ref{harnack} now.
\begin{proof}[Proof of Theorem \ref{harnack}] Without loss of generality, we may assume $f\geq 0$ since $\big|\tilde{P}_tf\big|\leq \tilde{P}_t|f|$. Let $\theta_\epsilon(r)=(r+\epsilon)^p$ for $p>1$, $r\geq 0$ and $\epsilon\in (0,1]$. In addition, assume at the moment $f\in \Lip_b(X)\cap \mathcal{D}(\mathcal{E})\cap L^1(X,\mu)$. Let $\gamma: [0,1]\rightarrow X$ be a Lipschitz continuous curve connecting $x$ and $y$ such that $\gamma_0=x$ and $\gamma_1=y$. Let $t>0$ and set
$$\alpha(s)=\frac{e^{2Ks}-1}{e^{2Kt}-1},\quad s\in [0,t].$$
Then $\alpha(0)=0$ and $\alpha(t)=1$. Let $\tilde{\gamma}_r=\gamma_{\alpha(r)}$ for $r\in [0,t]$. For $s\in (0,t)$ and $r\in [0,t]$, set
\begin{eqnarray*}\label{main1}
G(r,s)=-\log \int_X \theta_\epsilon(P_{t-s}f)\,\d (h_s\delta_{\tilde{\gamma}_r}).
\end{eqnarray*}

On the one hand, by Lemma \ref{lemmaofharnack1} with $\nu$ there replaced by $\delta_{\tilde{\gamma}_r}$, for any $r\in [0,t]$, the map $s\mapsto G(r,s)$ is continuous in $[0,t]$ and continuously differentiable in $(0,t)$ with
\begin{eqnarray}\label{main2}
\frac{\partial}{\partial s}G(r,s)=-\frac{\int_X \theta''_\epsilon(P_{t-s}f)\Gamma(P_{t-s}f)\,\d (h_s\delta_{\tilde{\gamma}_r})}{\tilde{P}_s\left(\theta_\epsilon(P_{t-s}f)\right)(\tilde{\gamma}_r)},
\end{eqnarray}
which immediately shows that the maps $G(r,\cdot)$ are Lipschitz continuous in $[0,t]$ with uniform Lipschitz constant w.r.t. $r\in [0,t]$, according to \eqref{L2gradient} and the fact $\sqrt{\Gamma(f)}\leq |D f|$ $\mu$-a.e. in $X$ (see e.g. \cite[Theorem 3.12]{AmbrosioGigliSavare2012}). On the other hand, noting that $\tilde{\gamma}$ is Lipschitz and by \eqref{w2contraction},
$$W_1(h_t\delta_x,h_t\delta_y)\leq e^{-Kt}d(x,y),\quad\mbox{for all }x,y\in X\mbox{ and }t\geq 0;$$
hence we deduce that the map $r\mapsto h_s\delta_{\tilde{\gamma}_r}$ is Lipschitz continuous in $[0,t]$ w.r.t. the $L^1$-Wasserstein distance $W_1$ uniformly in $s\in [0,t]$. By the fact that $\left\{\theta_\epsilon(P_{t-s}f)\right\}_{s\in [0,t]}$ is equi-Lipschitz, we know that the maps $G(\cdot,s)$ are Lipschitz continuous in [0,t] with uniform Lipschitz constant w.r.t. $s\in [0,t]$. Thus, the map $s\mapsto G(s,s)$ is Lipschitz continuous in $[0,t]$.

Let
$$I_1=\lim_{u\downarrow 0} \frac{G(s,s+u)-G(s,s)}{u},$$
and
$$I_2=\limsup_{u\downarrow 0}\frac{G(s,s)-G(s-u,s)}{u}.$$
Applying \cite[Lemma 4.3.4]{AmbrosioGigliSavare2005}, we get
\begin{eqnarray}\label{main3}
\frac{\d}{\d s}G(s,s)&\leq& I_1+I_2,\quad \mathcal{L}^1\mbox{-a.e. in } (0,t).
\end{eqnarray}

From \eqref{main2}, we easily get
\begin{eqnarray}\label{main3.1}
I_1&=&-\frac{\int_X \theta''_\epsilon(P_{t-s}f)\Gamma(P_{t-s}f)\,\d (h_s\delta_{\tilde{\gamma}_s})}{\tilde{P}_s\left(\theta_\epsilon(P_{t-s}f)\right)(\tilde{\gamma}_s)}\cr
&=&-\frac{p(p-1)\tilde{P}_s\big[(P_{t-s}f+\epsilon)^{p-2}\Gamma(P_{t-s}f)\big](\tilde{\gamma}_s)}
{\tilde{P}_s (\theta_\epsilon(P_{t-s}f))(\tilde{\gamma}_s)}.
\end{eqnarray}

For the estimate of $I_2$, it is also straightforward but needs some work. We claim:
\begin{eqnarray}\label{L1main}
\big|D\tilde{P}_s(\theta_\epsilon(P_{t-s}f))\big|\leq e^{-Ks} \tilde{P}_s\big[\theta'_\epsilon(P_{t-s}f)\sqrt{\Gamma(P_{t-s}f)}\,\big].
\end{eqnarray}
In fact, by Lemma \ref{SavareL1estimate} and the chain rule (which can be applied, since we can take $\theta_\epsilon(P_{t-s}f)-\theta_\epsilon(0)$ in place of $\theta_\epsilon(P_{t-s}f)$ without changing the local Lipschitz constant of \eqref{L1main}),
\begin{eqnarray*}\label{claim}
\sqrt{\Gamma\big(\tilde{P}_s(\theta_\epsilon(P_{t-s}f))\big)}\leq e^{-Ks} \tilde{P}_s\big[\theta'_\epsilon(P_{t-s}f)\sqrt{\Gamma(P_{t-s}f)}\,\big],\quad\mu\mbox{-a.e. in }X.
\end{eqnarray*}
Let $\xi$ be the right hand side of \eqref{L1main}. Since $f\in \Lip_b(X)\cap \mathcal{D}(\mathcal{E})\cap L^1(X,\mu)$, we have $\Gamma(P_{t-s}f)\in L^\infty(X,\mu)$, and hence $\xi\in \Lip_b(X)$ for $s>0$ by the regularizing property of $\tilde{P}_s$ (see e.g. \cite[Theorem 7.1]{AmbrosioGigliMondinoRajala} and the second paragraph in Section 3). Thus, by Lemma \ref{lemmaofharnack2}, we obtain \eqref{L1main}.

Applying \eqref{L1main}, we have
\begin{eqnarray}\label{main3.2}
I_2&=&\limsup_{u\downarrow 0}\frac{1}{u}\left[\log\int_X \theta_\epsilon(P_{t-s}f)\,\d (h_s\delta_{\tilde{\gamma}_{s-u}})-  \log\int_X \theta_\epsilon(P_{t-s}f)\,\d (h_s\delta_{\tilde{\gamma}_{s}})\right]\cr
&=&\limsup_{u\downarrow 0}\frac{1}{u}\left[\log \tilde{P}_s(\theta_\epsilon(P_{t-s}f))(\tilde{\gamma}_{s-u})- \log \tilde{P}_s(\theta_\epsilon(P_{t-s}f))(\tilde{\gamma}_{s})\right]\cr
&\leq &\frac{\big|D\tilde{P}_s(\theta_\epsilon(P_{t-s}f))\big|(\tilde{\gamma}_s)\, |\dot{\gamma}_{\alpha(s)}| \, |\alpha'(s)|}{\tilde{P}_s(\theta_\epsilon(P_{t-s}f))(\tilde{\gamma}_s)}\cr
&\leq & \frac{e^{-Ks}|\dot{\gamma}_{\alpha(s)}|\, |\alpha'(s)| \, \tilde{P}_s\big[\theta'_\epsilon(P_{t-s}f)\sqrt{\Gamma(P_{t-s}f)}\,\big](\tilde{\gamma}_s)}{\tilde{P}_s(\theta_\epsilon(P_{t-s}f))(\tilde{\gamma}_s)}.
\end{eqnarray}

Then combining \eqref{main3}, \eqref{main3.1} and \eqref{main3.2}, we get
\begin{eqnarray}\label{main4}
&&-\frac{\d}{\d s}G(s,s)\cr
&\geq& \frac{p(p-1)\tilde{P}_s\big[(P_{t-s}f+\epsilon)^{p-2}\Gamma(P_{t-s}f)\big](\tilde{\gamma}_s)}
{\tilde{P}_s(\theta_\epsilon(P_{t-s}f))(\tilde{\gamma}_s)}\cr
&& - \frac{e^{-Ks}|\dot{\gamma}_{\alpha(s)}|\, |\alpha'(s)|\tilde{P}_s\big[p (P_{t-s}f+\epsilon)^{p-1}\sqrt{\Gamma(P_{t-s}f)}\,\big](\tilde{\gamma}_s)}{\tilde{P}_s(\theta_\epsilon(P_{t-s}f))(\tilde{\gamma}_s)}\cr
&=& \frac{p\tilde{P}_s\left\{(P_{t-s}f+\epsilon)^p \left[(p-1)\frac{\Gamma(P_{t-s}f)}{(P_{t-s}f+\epsilon)^2} -e^{-Ks}|\dot{\gamma}_{\alpha(s)}| \alpha'(s)\frac{\sqrt{\Gamma(P_{t-s}f)}}{P_{t-s}f+\epsilon}  \right]   \right\}(\tilde{\gamma}_s)}{\tilde{P}_s(\theta_\epsilon(P_{t-s}f))(\tilde{\gamma}_s)}\cr
&\geq& \frac{p}{\tilde{P}_s(\theta_\epsilon(P_{t-s}f))(\tilde{\gamma}_s)}\tilde{P}_s\left[\theta_\epsilon(P_{t-s}f)\left(-\frac{e^{-2Ks}|\dot{\gamma}_{\alpha(s)}|^2 \alpha'(s)^2}{4(p-1)} \right) \right](\tilde{\gamma}_s)\cr
&=&-\frac{pe^{-2Ks}|\dot{\gamma}_{\alpha(s)}|^2 \alpha'(s)^2}{4(p-1)},\quad \mathcal{L}^1\mbox{-a.e. in } (0,t),
\end{eqnarray}
where in the last inequality we used the Cauchy--Schwarz inequality. Integrating both sides of \eqref{main4} w.r.t. $\d s$ in $(0,t)$ and minimizing w.r.t. $\gamma$ connecting $x,y$, we have
\begin{eqnarray*}
&&\log \tilde{P}_t\big( (f+\epsilon)^p\big)(y)- \log\big(\tilde{P}_tf+\epsilon \big)^p(x)\cr
&\geq& -\int_0^t \frac{pe^{-2Ks}\alpha'(s)^2 d(x,y)^2}{4(p-1)}\,\d s\\
&=&-\frac{p K d(x,y)^2}{2(p-1)(e^{2Kt}-1)}.
\end{eqnarray*}
Thus, we obtain \eqref{harnack1} for every $f\in \Lip_b(X)\cap \mathcal{D}(\mathcal{E})\cap L^1(X,\mu)$.

Now for general $f\in L^\infty(X,\mu)$ with $f\geq 0$,  we choose a uniformly bounded sequence $\{f_n\}_{n\geq 1}$ from $\Lip_b(X)\cap \mathcal{D}(\mathcal{E})\cap L^1(X,\mu)$ such that $f_n\rightarrow f$ $\mu$-a.e. as $n\rightarrow\infty$. Indeed, we can use the regularity property of  $P_t$ to construct $f_n$. Hence, $\tilde{P}_tf_n$ converges to $\tilde{P}_t f$ pointwise as $n\rightarrow \infty$. Thus, we get \eqref{harnack1} for $f\in L^\infty(X,\mu)$. Finally, by a truncation argument, \eqref{harnack1} holds for any $f\in L^1(X,\mu)+L^\infty(X,\mu)$.

Letting $\epsilon\downarrow 0$, we have \eqref{harnack2}.
\end{proof}

\section{Applications of the Harnack inequality}
In this section, we quickly show some consequences of the Harnack inequality which generalize the ones on the smooth Riemannian manifold and are natural to expect in our more abstract setting.

An immediate consequence of the Harnack inequality in Theorem \ref{harnack} is the log-Harnack inequality, which is also proved in \cite{AmbrosioGigliSavare2012} by the same idea used in the proof of \cite[Theorem 1.1(6)]{Wang2011}. Also, it can be proved in an elementary way; see \cite[Propositions 2.1 and 2.2]{Wang2010}.
\begin{proposition}\label{loghrnack}
Let $(X,d,\mu)$ be a $\RCD(K,\infty)$ space with  $K\in \R$. For $f\in L^1(X,\mu)+L^\infty(X,\mu)$ with $f\geq 0$, $t>0$ and $x,y\in X$, it holds
\begin{eqnarray}\label{logharnack1}
\tilde{P}_t\left(\log f\right)(x)\leq \log \big(\tilde{P}_tf \big)(y) + \frac{ K d(x,y)^2}{2\left(e^{2Kt}-1\right)}.
\end{eqnarray}
\end{proposition}

\begin{remark}
In fact, applying the log-Harnack inequality \eqref{loghrnack}, we can immediately obtain the strong Feller property of the semigroup $\tilde{P}_t$, i.e., $\tilde{P}_t L^\infty(X,\mu)\subset C_b(X)$ for any $t>0$, by an elementary calculus; see \cite[Propersition 2.3]{Wang2010}. However,  Ambrosio--Gigli--Savar\'{e} \cite[Theorem 3.17]{AmbrosioGigliSavare2012} showed the stronger one, i.e., $\tilde{P}_t L^\infty(X,\mu)\subset \Lip_b(X)$ for any $t>0$, with a different approach.
\end{remark}

Recall \eqref{mass preserving} that the heat flow $\{P_t\}_{t\geq 0}$ possesses the mass preserving property w.r.t. $\mu$, the log-Harnack inequality \eqref{logharnack1} implies the following entropy-cost inequality. The proof is simple and it is verbatim the one on the smooth Riemannian manifold (see e.g. \cite[Thorem 2.4.1]{Wang2013a}).
\begin{corollary}\label{transport}
Let $(X,d,\mu)$ be a $\RCD(K,\infty)$ space with  $K\in \R$ and $\mu$ being a probability measure. Then
$$\mu\big[\big(\tilde{P}_tf^2\big)\log\big(\tilde{P}_tf^2\big)\big]\leq \frac{KW_2(f^2\mu,\mu)^2}{2\left(e^{2Kt}-1\right)},\quad\mu\big(f^2\big)=1.$$
\end{corollary}

\begin{proof} Suppose that $\mu\big(f^2\big)=1$. Applying the log-Harnack inequality of Proposition \ref{loghrnack} to $\tilde{P}_tf^2$ in place of $f$, we obtain for any $x,y\in X$,
$$\tilde{P}_t\big(\log \tilde{P}_tf^2\big)(x)\leq \log\big(\tilde{P}_{2t}f^2\big)(y)+ \frac{ K d(x,y)^2}{2\left(e^{2Kt}-1\right)}.$$
Integrating w.r.t. the optimal coupling of $f^2\mu$ and $\mu$ for the $L^2$-transportation cost, we have
\begin{eqnarray*}
\mu\big[\big(\tilde{P}_tf^2\big)\log\big(\tilde{P}_tf^2\big)\big]&\leq& \frac{K W_2(f^2\mu,\mu)^2}{2\left(e^{2Kt}-1\right)}+
\mu\big[\log\big(\tilde{P}_{2t}f^2\big)\big]\\
&\leq& \frac{KW_2(f^2\mu,\mu)^2}{2\left(e^{2Kt}-1\right)},
\end{eqnarray*}
where the second inequality follows from the Jensen inequality and the mass preserving property \eqref{mass preserving}.
\end{proof}

For any $x\in X$ and $t>0$, since $h_t(\delta_x)\ll \mu$, $h_t(\delta_x)$ has a density w.r.t. $\mu$, denoted by $q_t[x]$, such that the map $(x,y)\in X\times X\mapsto q_t[x](y)$ can be chosen to be $\mu\times\mu$-measurable, and
$$\tilde{P}_tf(x)=\int_Xf(y)q_t[x](y)\,\d\mu(y),$$
for any $\mu$-measurable and semi-integrable function $f$. It can be shown that $q_t[x](y)=q_t[y](x)$ $\mu\times\mu$-a.e. in $X\times X$, for all $t>0$. Hence, we denote it  by $\tilde{p}_t(x,y)$ and call it the heat kernel of the semigroup $\tilde{P}_t$. See \cite[Theorem 7.1]{AmbrosioGigliMondinoRajala} for more properties of $\tilde{p}_t$.

In particular, when $\mu$ is a probability measure, the symmetry and the log-Harnack inequality \eqref{loghrnack} immediately imply that for every $t>0$ and $x\in X$ (see \cite[Corollary 4.7]{AmbrosioGigliSavare2012}),
\begin{eqnarray*}\label{positivitykernel}
\tilde{p}_t(x,y)\geq \exp\left\{-\frac{K d(x,y)^2}{2(e^{Kt}-1)}\right\},\quad\mbox{for }\mu\mbox{-a.e. }y\in X.
\end{eqnarray*}

Wang's theorem \cite[Theorem 1.1]{Wang2001} expresses roughly that, on the smooth Riemannian manifold, a strong enough Gaussian concentration implies the log-Sobolev inequality when the Ricci curvature is bounded from below by a non-positive constant. Now we are ready to generalize it to the $\RCD(K,\infty)$ space and prove it following the same idea of the aforementioned reference.

Given $p,q\in [1,\infty]$, define the operator norm of $\tilde{P}_t$ by
$$\|\tilde{P}_t\|_{p\rightarrow q}=\sup_{f\in L^p\cap L^2(X,\mu)\setminus \{0\}} \frac{\|\tilde{P}_t f\|_{L^q(X,\mu)}}{\|f\|_{L^p(X,\mu)}}.$$
In fact, the Markovian property allows to extend $\tilde{P}_t$ to an operator in $L^p(X,\mu)$ so that the range $L^p\cap L^2(X,\mu)$ of $f$ can be replaced by $L^p(X,\mu)$.
\begin{theorem}\label{wangcriteria}
Let $(X,d,\mu)$ be a $\RCD(-K,\infty)$ space with  $K\geq 0$ and $\mu$ being a probability measure. Assume that there exists $\epsilon>0$ such that
\begin{eqnarray}\label{integral}
\mu\big(e^{(K/2+\epsilon)d(o,\cdot)^2}\big)<\infty,\quad\mbox{for some fixed point }o\in X.
\end{eqnarray}
Then the log-Sobolev inequality
\begin{eqnarray}\label{logsobolev}
\mu(f^2\log f^2)\leq C\mathcal{E}(f,f),\quad f\in \mathcal{D}(\mathcal{E}),\,\ \mu(f^2)=1
\end{eqnarray}
holds for some constant $C>0$.
\end{theorem}
\begin{proof}
By \eqref{harnack2}, for any $f\in L^1(X,\mu)$, $p>1$, $t\geq 0$ and $x,y\in X$, it holds
\begin{eqnarray}\label{harnack3}
\big|\tilde{P}_tf(x) \big|^p\leq \tilde{P}_t|f|^p(y)\exp\left[\frac{p K d(x,y)^2}{2(p-1)(1-e^{-2Kt})}\right].
\end{eqnarray}
For any $p>2$, $f\geq 0$ with $\mu(f^p)=1$ and some fixed point $o\in X$, by \eqref{harnack3} and the mass preserving property \eqref{mass preserving}, we have
\begin{eqnarray*}
1=\mu(\tilde{P}_tf^p)&\geq& \big(\tilde{P}_tf(x)\big)^p\int_X\exp\left[-\frac{p K d(x,y)^2}{2(p-1)(1-e^{-2Kt})}\right]\,\d\mu(y)\\
&\geq& \big(\tilde{P}_tf(x)\big)^p \mu(B(o,1))\exp\left[-\frac{p K (d(o,x)+1)^2}{2(p-1)(1-e^{-2Kt})}\right].
\end{eqnarray*}
Note that $\mu(B(o,1))>0$. Hence, there exists a constant $c>0$ such that for any $t\geq 0$ and $x\in X$,
\begin{eqnarray}\label{wang}
\tilde{P}_tf(x)\leq c\exp\left[\frac{K (d(o,x)+1)^2}{2(p-1)(1-e^{-2Kt})}\right].
\end{eqnarray}
Combining \eqref{wang} with \eqref{integral}, we get
\begin{eqnarray*}
\mu\big((\tilde{P}_tf)^{p+\delta}\big)\leq c^{p+\delta}\int_X\exp\left[\frac{(p+\delta) K (d(o,\cdot)+1)^2}{2(p-1)(1-e^{-2Kt})}\right]\,\d\mu<\infty,
\end{eqnarray*}
for small $\delta>0$ and big $p$ and $t$. Let $\sigma=\frac{p-2}{2(p-1)}$, which belongs to $(0,\frac{1}{2})$, and let $r>2$ such that $\frac{1}{r}=\sigma+\frac{1-\sigma}{p+\delta}$. Since $\{\tilde{P}_t\}_{t\geq 0}$ is a contraction semigroup of linear operators on $L^1(X,\mu)$, by the Riesz-Thorin interpolation theorem (see e.g.  \cite[Section 1.1.5]{Davies1989}), we obtain $\|\tilde{P}_t\|_{2\rightarrow r}<\infty$ for some $t>0$, which implies that the defective log-Sobolev inequality
\begin{eqnarray}\label{DLSI}
\mu(f^2\log f^2)\leq C_1\mathcal{E}(f,f)+C_2,\quad f\in \mathcal{D}(\mathcal{E}),\,\ \mu(f^2)=1
\end{eqnarray}
holds for some constants $C_1, C_2>0$ (see e.g. \cite{Gross1993}). As we know the heat kernel $\tilde{p}_t(x,y)>0$ for $\mu\times\mu$-a.e. $(x,y)\in X\times X$, by \cite[Lemma 2.6]{Aida1998}, $\tilde{P}_t$ is ergodic which is equivalent to the irreducibility of $(\mathcal{E}, \mathcal{D}(\mathcal{E}))$, i.e., $\mathcal{E}(f,f)=0$ implying $f$ is constant, in this setting.  Recall that $(\mathcal{E}, \mathcal{D}(\mathcal{E}))$ is conservative, i.e., $\tilde{P}_t1=1$. Applying \cite[Corollary 1.3]{Wang2013c}, we obtain the desired log-Sobolev inequality \eqref{logsobolev}.
\end{proof}

\begin{remark}
Indeed, when $K>0$, the log-Sobolev inequality \eqref{logsobolev} was obtained on the $\RCD(K,\infty)$ space $(X,d,\mu)$ with $\mu$ being a probability measure; see (6.10) in \cite{AmbrosioGigliSavare2011b} for example.
\end{remark}

Recall that $\tilde{P}_t$ is called ultracontractive if $\|\tilde{P}_t\|_{2\rightarrow \infty}<\infty$ for all $t>0$, supercontractive if $\|\tilde{P}_t\|_{2\rightarrow 4}<\infty$ for all $t>0$, and hypercontractive if $\|\tilde{P}_t\|_{2\rightarrow 4}<\infty$ for some $t>0$. These contractivity properties were studied extensively via various functional inequalities and heat kernel upper bounds; see e.g. \cite{BakryEmery,DaviesSimon1984,Davies1989,Gross1993,Wang2005} and references therein. Since $\tilde{P}_t$ is symmetric, it is easy to observe that $\|\tilde{P}_t\|_{1\rightarrow \infty}=\|\tilde{P}_{t/2}\|^2_{2\rightarrow \infty}$ for all $t>0$ (see e.g. \cite[Lemma 2.1.2]{Davies1989} for a simple proof), and hence $\|\tilde{P}_t\|_{1\rightarrow \infty}<\infty$ for all $t>0$ is equivalent to $\|\tilde{P}_t\|_{2\rightarrow \infty}<\infty$ for all $t>0$. We mention here that, under the ultracontractivity assumption of $\tilde{P}_t$, the global Lipschitz continuity of the heat kernel $p_t(x,y)$ w.r.t. the spatial variables $x$ and $y$ is obtained in \cite[Proposition 6.4]{AmbrosioGigliSavare2011b}.

The following results is an analogous of the one on the smooth Riemannian manifold due to R\"{o}ckner--Wang \cite{RochnerWang2003}. The idea of its proof is the same to the smooth case and we omit it here to keep the note short (see also the proof of Theorem \ref{wangcriteria}).
\begin{corollary}\label{ultra}
Let $(X,d,\mu)$ be a $\RCD(K,\infty)$ space with  $K\in \R$ and $\mu$ being a probability measure. Then, for some (hence any) fixed point $o$ in $X$,
\begin{itemize}
\item[\rm(1)] $\tilde{P}_t$ is ultracontractive if and only if $\|\tilde{P}_t\exp [\lambda d(o,\cdot)^2] \|_{L^\infty(X,\mu)}<\infty$ for any $t,\lambda>0$, and, moreover, there exist some constants $\lambda_1, \lambda_2>0$ such that
\begin{eqnarray}\label{ultra-1}
\|\tilde{P}_t\|_{1\rightarrow \infty}\leq e^{\lambda_1/t}\|\tilde{P}_{t/4}\exp\big(\lambda_2 d(o,\cdot)^2/t\big)\|_{L^\infty(X,\mu)},\quad\mbox{for any }t>0;
\end{eqnarray}

\item[\rm(2)] $\tilde{P}_t$ is supercontractive if and only if $\mu(\exp [\lambda d(o,\cdot)^2] )<\infty$ for any $\lambda>0$;

\item[\rm(3)] if there exists a constant $\lambda> -K/2$ such that $\mu(e^{\lambda d(o,\cdot)^2})<\infty$, then $\tilde{P}_t$ is hypercontractive.
\end{itemize}
\end{corollary}

\medskip

{\bf Acknowledgements.} The author would like to thank Dr. Dejun Luo for his careful reading of the first version of this work in which the reference measure is assumed to be a probability measure. And thanks are also given to an anonymous referee for his or her careful reading and meticulous comments, which make the presentation of this paper better.

\end{document}